\documentclass[letterpaper,11pt,twoside,keywordsasfootnote,addressatend,noinfoline]{article}
\usepackage{fullpage}
\usepackage[english]{babel}
\usepackage{amssymb}
\usepackage{amsmath}
\usepackage{theorem}
\usepackage{epsfig}
\usepackage{enumitem}
\usepackage{mathrsfs}
\usepackage{imsart}
\usepackage{color}

\theorembodyfont{\slshape}
\newtheorem{theorem}{Theorem}

\newtheorem{lemma}[theorem]{Lemma}

\newenvironment{proof}{\noindent{\scshape Proof.}}{\hspace*{2mm} $\square$}
\newenvironment{proofof}[1]{\noindent{\scshape{Proof of #1.}}}{\hspace*{2mm}~$\square$}

\newcommand{\G}{\mathscr{G}}
\newcommand{\V}{\mathscr{V}}
\newcommand{\E}{\mathscr{E}}

\newcommand{\C}{\mathscr{C}}
\newcommand{\N}{\mathbb{N}}

\newcommand{\ind}{\mathbf{1}}

\newcommand{\n}{\hspace*{-6pt}}

\DeclareMathOperator{\card}{card \,}
\DeclareMathOperator{\bernoulli}{Bernoulli \,}
\DeclareMathOperator{\uniform}{Uniform \,}

\DeclareMathOperator{\binomial}{Binomial \,}

\DeclareMathOperator{\var}{Var}

\begin{document}

\begin{frontmatter}
\title     {Cluster size in bond percolation \\ on the Platonic solids}
\runtitle  {Cluster size in bond percolation on the Platonic solids}
\author    {Nicolas Lanchier\thanks{Nicolas Lanchier was partially supported by NSF grant CNS-2000792.} and Axel La Salle}
\runauthor {Nicolas Lanchier and Axel La Salle}
\address   {School of Mathematical and Statistical Sciences \\ Arizona State University \\ Tempe, AZ 85287, USA.}

\maketitle

\begin{abstract} \ \
 The main objective of this paper is to study the size of a typical cluster of bond percolation on each of the five Platonic solids:
 the tetrahedron, the cube, the octahedron, the dodecahedron and the icosahedron.
 Looking at the clusters from a dynamical point of view, i.e., comparing the clusters with birth processes, we first prove that the first and second
 moments of the cluster size are bounded by their counterparts in a certain branching process, which results in explicit upper bounds that are accurate
 when the density of open edges is small.
 Using that vertices surrounded by closed edges cannot be reached by an open path, we also derive upper bounds that, on the contrary, are accurate when
 the density of open edges is large.
 These upper bounds hold in fact for all regular graphs.
 Specializing in the five~Platonic solids, the exact value of (or lower bounds for) the first and second moments are obtained from the inclusion-exclusion
 principle and a computer program.
 The goal of our program is not to simulate the stochastic process but to compute exactly sums of integers that are too large to be computed by hand so
 these results are analytical, not numerical.
\end{abstract}



\end{frontmatter}


\section{Introduction}
\label{sec:intro}
 Bond percolation on a simple undirected graph is a collection of independent Bernoulli random variables with the same success probability~$p$ indexed by the
 set of edges, with the edges associated to a success being referred to as open edges, and the ones associated to a failure being referred to as closed edges.
 The open cluster containing a vertex~$x$ is the random subset of vertices that are connected to vertex~$x$ by a path of open edges.
 This stochastic model was introduced in~\cite{broadbent_hammersley_1957} to study the random spread of a fluid through a medium. \\
\indent Bond percolation is traditionally studied on infinite graphs such as the~$d$-dimensional integer lattice in which case the quantity of interest is
 the percolation probability, the probability that the cluster of open edges containing the origin is infinite.
 For bond percolation on integer lattices, it follows from Kolmogorov's zero-one law that the existence of an infinite cluster of open edges is an event that
 has probability either zero or one.
 This, together with a basic coupling argument, implies that there is a phase transition at a critical value~$p_c$ for the density of open edges from a
 subcritical phase where all the open clusters are almost surely finite to a supercritical phase where there is at least one infinite cluster of open edges.
 It is known that the cluster size decays exponentially in the subcritical phase~\cite{menshikov_1986} and that there is a unique infinite cluster of open
 edges called the infinite percolation cluster in the supercritical phase~\cite{aizenman_kesten_newman_1987a, aizenman_kesten_newman_1987b}.
 Using planar duality, coupling arguments and the uniqueness of the infinite percolation cluster, it can also be proved that the critical value in two
 dimensions is equal to one-half~\cite{kesten_1980}.
 We refer the interested reader to~\cite{grimmett_1999} for additional results about bond percolation on integer lattices,
 and to~\cite[chapter~13]{Lanchier_2017} for a brief overview. \\
\indent Bond percolation has also been studied on fairly general finite connected graphs~\cite{alon_benjamini_stacey_2004}.
 Important particular cases are the complete graph, in which case the set of open edges form the very popular~Erd\H{o}s-R\'{e}nyi random
 graph~\cite{erdos_renyi_1959}, as well as the hypercube~\cite{ajtai_komlos_szemeredi_1982}.
 Due to the finiteness of the underlying graph, all the open clusters are finite so whether there exists an infinite percolation cluster or not becomes
 irrelevant.
 Such processes, however, still exhibit a phase transition in the sense that, in the limit as the number of vertices goes to infinite, there is a giant
 component of open edges (an open cluster whose size scales like the size of the graph) if and only if~$p$ is exceeds a certain critical value.
 In particular, most of the works about bond percolation on finite graphs is concerned with asymptotics in the large graph limit. \\
\indent In contrast, the objective of this paper is to study~(the first and second moments of) the size distribution of a typical cluster of bond
 percolation on each of the five Platonic solids: the tetrahedron, the cube, the octahedron, the dodecahedron and the icosahedron.
 The motivation originates from our previous works~\cite{jevtic_lanchier_2020, jevtic_lanchier_lasalle_2020} that introduce a mathematical framework based
 on Poisson processes, random graphs equipped with a cost topology and bond percolation to model the aggregate loss resulting from cyber risks.
 Insurance premiums are based on the mean and variance of the aggregate loss which, in turn, can be easily expressed using the first and second
 moments of the size of the percolation clusters.
 Estimates for the size of the clusters are given in~\cite{jevtic_lanchier_2020} for the process on finite random trees
 and in~\cite{jevtic_lanchier_lasalle_2020} for the process on path, ring and star graphs.
 Even though our present work does not have any applications in the field of cyber insurance (because the Platonic solids are not realistic models
 of insurable networks), studying the size of percolation clusters on the Platonic solids is a very natural question in probability theory.


\section{Main results}
\label{sec:results}
\begin{figure}[t!]
\centering
\scalebox{0.15}{\input{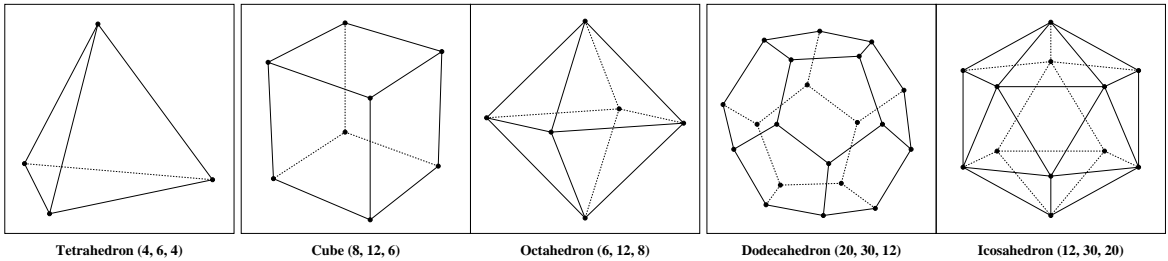}}
\caption{\upshape{Picture of the five Platonic solids.
                  The numbers between parentheses refer to the number of vertices, the number of edges, and the number of faces, respectively.
                  Note that the tetrahedron is dual to itself, the cube and the octahedron are dual to each other, and the dodecahedron and icosahedron
                  are dual to each other.}}
\label{fig:graphs}
\end{figure}
\begin{figure}[t!]
\centering
\scalebox{0.20}{\input{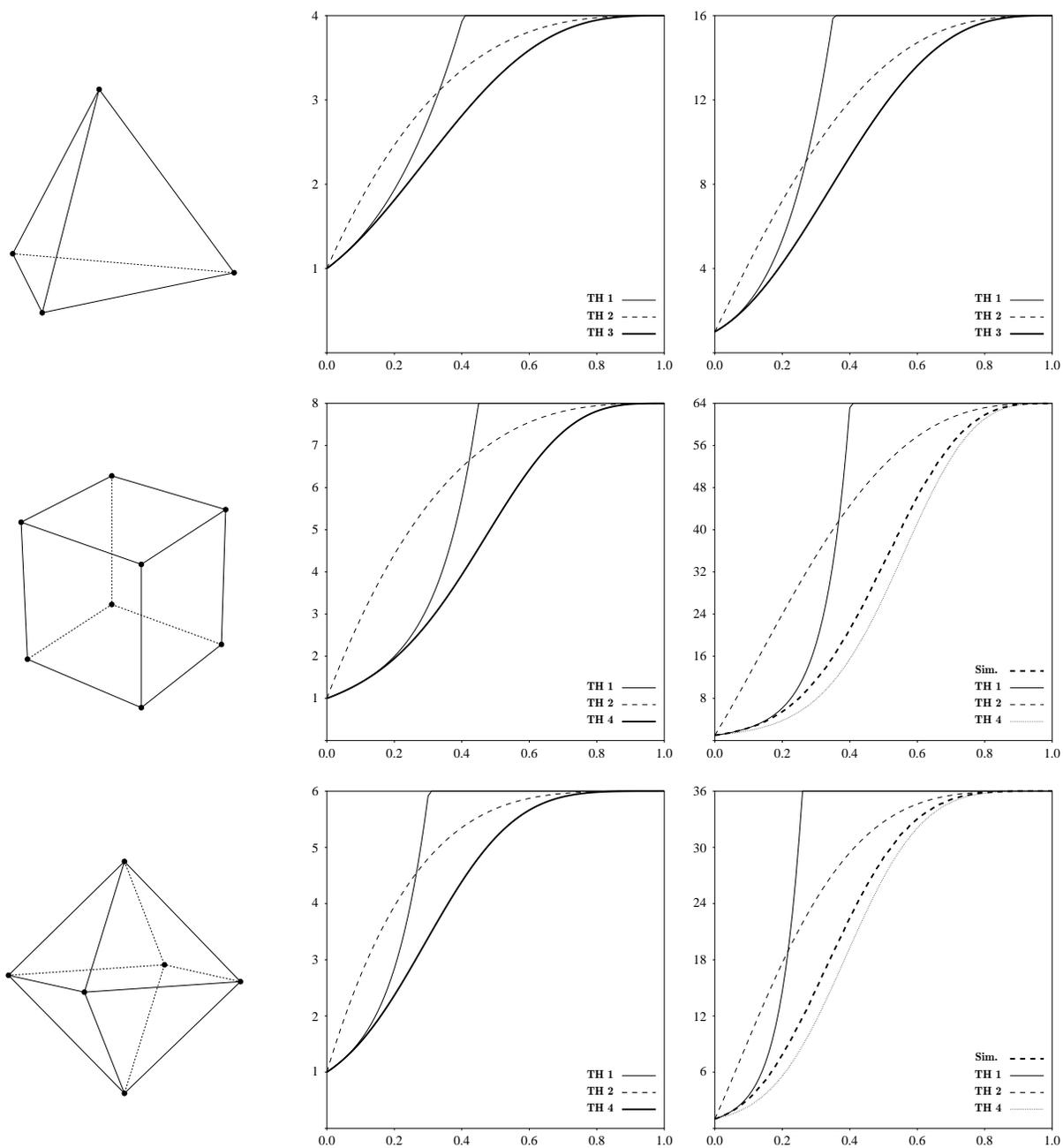}}
\caption{\upshape{First moment on the left and second moment on the right of the size distribution of bond percolation clusters on the tetrahedron (top),
                  cube (middle) and octahedron (bottom) as functions of the probability~$p$.
                  The thick solid lines show the exact expressions in~TH~\ref{th:04} and~\ref{th:06-08}, the thick dashed lines show the second moment obtained
                  from the average of one hundred thousand independent realizations of the process for various values of~$p$, and the other curves show the upper
                  bounds in TH~\ref{th:branching} and TH~\ref{th:plarge} for the appropriate values of~$D$ and~$N$.}}
\label{fig:moment-04-06-08}
\end{figure}
\begin{figure}[t!]
\centering
\scalebox{0.20}{\input{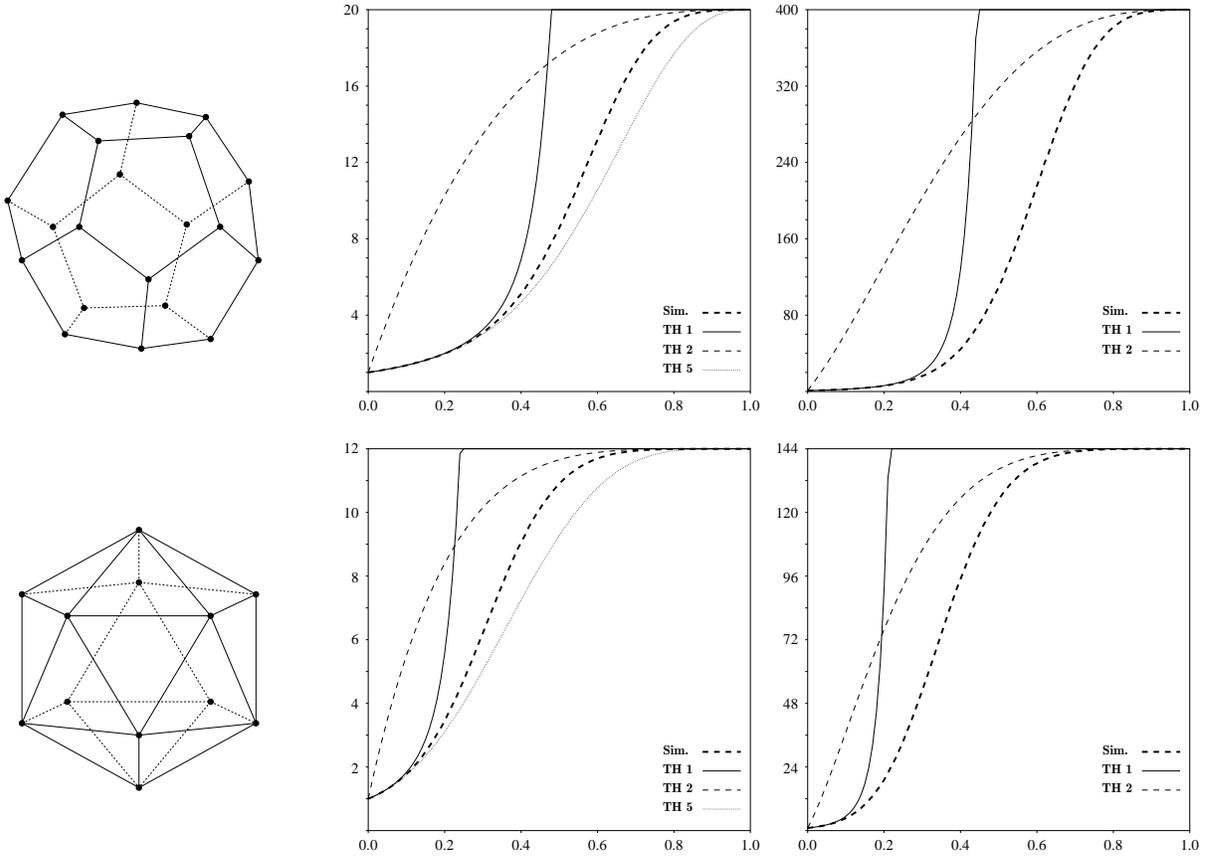}}
\caption{\upshape{First moment on the left and second moment on the right of the size distribution of bond percolation clusters on the dodecahedron (top)
                  and icosahedron (bottom) as functions of the probability~$p$.
                  The thick dashed lines show the second moment obtained from the average of one hundred thousand independent realizations of the process for various
                  values of~$p$ while the other curves show the upper bounds in TH~\ref{th:branching} and TH~\ref{th:plarge} for the appropriate values of
                  the degree~$D$ and the number of vertices~$N$, and the lower bounds in TH~\ref{th:12-20}.}}
\label{fig:moment-12-20}
\end{figure}
 Having a simple undirected graph~$\G = (\V, \E)$, let
 $$ x = \uniform (\V) \quad \hbox{and} \quad \zeta (e) = \bernoulli (p), e \in \E $$
 be a vertex chosen uniformly at random and a collection of Bernoulli random variables with the same success probability~$p$ on the set of edges.
 The edges with~$\zeta (e) = 1$ are said to be open while the edges with~$\zeta (e) = 0$ are said to be closed, and we let
 $$ \C_x = \{y \in \V : \hbox{there is a path of open edges connecting~$x$ and~$y$} \} $$
 be the percolation cluster containing~$x$.
 The main objective of this paper is to study the first and second moments of~$S = \card (\C_x)$ = the size of this percolation cluster when the graph~$\G$
 consists of each of the five Platonic solids depicted in Figure~\ref{fig:graphs}.
 Our first result gives upper bounds for the first and second moments of the cluster size that apply to all finite regular graphs and are not
 restricted to the Platonic solids.
 The idea is to think of the cluster~$\C_x$ as a dynamical object described by a birth process starting with one particle at~$x$ and in which particles
 give birth with probability~$p$ onto vacant adjacent vertices.
 The size of the cluster is equal to the ultimate number of particles in the birth process which, in turn, is dominated stochastically by the number of
 individuals up to generation~$\card (\V) - 1$ in a certain branching process.
 Computing the first and second moments of the number of individuals in the branching process gives the following upper bounds.
\begin{theorem} --
\label{th:branching}
 For every~$D$-regular graph with~$N$ vertices,
 $$ \begin{array}{rcl}
        E (S) & \n \leq \n & \displaystyle 1 + Dp \bigg(\frac{1 - \nu^R}{1 - \nu} \bigg) \vspace*{8pt} \\
      E (S^2) & \n \leq \n & \displaystyle \bigg(1 + Dp \bigg(\frac{1 - \nu^R}{1 - \nu} \bigg) \bigg)^2 + \frac{Dp (1 - p)}{(1 - \nu)^2} \bigg(\frac{(1 - \nu^R)(1 + \nu^{R + 1})}{1 - \nu} - 2R \nu^R \bigg) \end{array} $$
 where~$\nu = (D - 1) p$ and~$R = N - 1$.
\end{theorem}
 Taking~$D$ and~$N$ in the theorem to be the degree and the number of vertices in each of the Platonic solids, we get the solid curves
 in Figures~\ref{fig:moment-04-06-08} and~\ref{fig:moment-12-20}.
 Note that these upper bounds are only accurate for~$p$ small.
 To have upper bounds that are accurate for~$p$ large, we simply use that a vertex~$y \neq x$ cannot be in the percolation cluster~$\C_x$ when all the edges
 incident to~$x$ are closed.
 This gives the following result that again applies to all finite regular graphs.
\begin{theorem} --
\label{th:plarge}
 For every~$D$-regular graph with~$N$ vertices,
 $$ \begin{array}{rcl}
        E (S) & \n \leq \n & N - (N - 1)(1 - p)^D \vspace*{4pt} \\
      E (S^2) & \n \leq \n & N^2 - (N - 1)(2N - 1)(1 - p)^D + (N - 1)(N - 2)(1 - p)^{2D - 1}. \end{array} $$
\end{theorem}
 Taking~$D$ and~$N$ in the theorem to be the degree and the number of vertices in each of the Platonic solids, we get the dashed curves
 in Figures~\ref{fig:moment-04-06-08} and~\ref{fig:moment-12-20}. \\
\indent Our last results are specific to the five Platonic solids and we denote by~$S_f$ the size of a percolation cluster on the solid with~$f$ faces.
 To explain these results, we first observe that the mean cluster size can be easily expressed using the probability that each vertex belongs to the open
 cluster~$\C_x$ which, in turn, is equal to the probability that at least one of the self-avoiding paths connecting~$x$ to this vertex is open.
 In particular, identifying all the self-avoiding paths connecting~$x$ to any other vertex and using the inclusion-exclusion identity give an exact
 expression for the first moment.
 The same holds for the second moment looking instead at all the pairs of paths connecting~$x$ to two other vertices.
 This approach also shows that the first and second moments of the cluster size are polynomials in~$p$ with integer coefficients and degree (at most) the
 total number of edges so, to state our next results and shorten the notation, we let
 $$ P_k = (p^0, p^1, p^2, \ldots, p^k)^T \quad \hbox{for all} \quad k \in \N. $$
 The main difficulties following this strategy is to identify all the self-avoiding paths and compute the probability that any sub-collection of
 paths are simultaneously open.
 Recall that, when dealing with~$n$ events, the inclusion-exclusion identity consists of a sum of~$2^n - 1$ terms.
 For the tetrahedron, the moments of the cluster size are polynomials with degree six, and there are five self-avoiding paths connecting any two vertices,
 and ten pairs of self-avoiding paths connecting any three vertices, therefore the number of terms in the inclusion-exclusion identity are
 $$ 2^5 - 1 = 31 \ \ \hbox{for the first moment} \quad \hbox{and} \quad 2^{10} - 1 = 1,023 \ \ \hbox{for the second moment}. $$
 In particular, we compute the first moment by hand whereas for the second moment we rely on a computer program that returns the exact value of
 the (seven) coefficients.
\begin{theorem}[tetrahedron] --
\label{th:04}
 For all~$p \in (0, 1)$,
 $$ E (S_4) = (1, 3, 6, 0, -21, 21, -6) \cdot P_6 \quad \hbox{and} \quad E (S_4^2) = (1, 9, 36, 30, -171, 153, -42) \cdot P_6. $$
\end{theorem}
 The cube and the octahedron both have twelve edges.
 There are respectively
\begin{itemize}
 \item 15, 16, 18 self-avoiding paths connecting two vertices at distance 1, 2, 3 on the cube, \vspace*{4pt}
 \item 26, 28 self-avoiding paths connecting two vertices at distance 1, 2 on the octahedron.
\end{itemize}
 In particular, the first moment of the cluster size cannot be computed by hand for the cube and the octahedron because the number of terms in the
 inclusion-exclusion identity ranges from tens of thousands to hundreds of millions.
 Identifying all these paths and using the same computer program as before, we get the following theorem.
\begin{theorem}[cube and octahedron] --
\label{th:06-08}
 For all~$p \in (0, 1)$,
 $$ \begin{array}{rcl}
      E (S_6) & \n = \n & (1, 3, 6, 12, 9, 12, -81, - 75, 69, 473, - 777, 447, -91) \cdot P_{12} \vspace*{4pt} \\
      E (S_8) & \n = \n & (1, 4, 12, 20, -14, -196, 12, 1316, -2815, 2824, -1564, 464, -58) \cdot P_{12}. \end{array} $$
\end{theorem}
 The dodecahedron and the icosahedron both have thirty edges.
 For these two solids, even writing down all the self-avoiding paths connecting two vertices is beyond human capability so we only focus on the paths
 of length at most five for the dodecahedron and of length at most three for the icosahedron.
 Using that two vertices are in the same open cluster if (but not only if) at least one of the paths is open, together with the inclusion-exclusion
 identity and our computer program, we get the following lower bounds for the mean cluster size.
\begin{theorem}[dodecahedron and icosahedron] --
\label{th:12-20}
 For all~$p \in (0, 1)$,
 $$ \begin{array}{rcl}
      E (S_{12}) & \n \geq \n & (1, 3, 6, 12, 24, 30, -24, -30, -36, 3, -6, 42, \vspace*{2pt} \\ && \hspace*{20pt}
                                -6, 18, -21, 14, 0, -6, -9, 0, 0, 6, 0, 0, -1, 0, 0, 0, 0, 0, 0) \cdot P_{30} \vspace*{4pt} \\
      E (S_{20}) & \n \geq \n & (1, 5, 20, 60, -90, -75, 0, 190, -10, -80, -60, 10, \vspace*{2pt} \\ && \hspace*{20pt}
                                -5, 120, -35, -88, 35, 40, -35, 10, -1, 0, 0, 0, 0, 0, 0, 0, 0, 0, 0) \cdot P_{30}. \end{array} $$
\end{theorem}
 The first and second moments in Theorem~\ref{th:04} and the first moments in Theorem~\ref{th:06-08} are represented
 by the thick solid curves in Figure~\ref{fig:moment-04-06-08}.
 These curves fit perfectly with numerical solutions obtained from one hundred thousands independent realizations of the percolation process.
 The lower bounds for the first moments in Theorem~\ref{th:12-20} are represented by the dotted curves in Figure~\ref{fig:moment-12-20}.


\section{Proof of Theorem~\ref{th:branching} (branching processes)}
\label{sec:branching}
 This section is devoted to the proof of Theorem~\ref{th:branching}.
 Though our focus is on the Platonic solids, we recall that the theorem applies to every finite~$D$-regular graph~$\G = (\V, \E)$.
 The basic idea of the proof is to use a coupling argument to compare the size of the percolation cluster starting at a given vertex with the number of
 individuals in a certain branching process. \vspace*{4pt}


\noindent{\bf Birth process.}
 Having a vertex~$x \in \V$ and a realization of bond percolation with parameter~$p$ on the graph, we consider the following discrete-time
 birth process~$(\xi_n)$.
 The state at time~$n$ is a spatial configuration of particles on the vertices:
 $$ \xi_n \subset \V \quad \hbox{where} \quad \xi_n = \hbox{set of vertices occupied by a particle at time~$n$}. $$
 The process starts at generation~0 with a particle at~$x$, i.e., $\xi_0 = \{x \}$.
\begin{itemize}
 \item For each vertex~$y$ adjacent to vertex~$x$, the particle at~$x$ gives birth to a particle sent to vertex~$y$ if and only if edge~$(x, y)$ is open.
\end{itemize}
 The children of the particle at~$x$ are called the particles of generation~1.
 Assume that the process has been defined up to generation~$n > 0$, and let
 $$ Y_n = \card (\xi_n \setminus \xi_{n - 1}) $$
 be the number of particles of that generation.
 Label arbitrarily~$1, 2, \ldots, Y_n$ the particles of generation~$n$ and let~$x_{n, 1}, x_{n, 2}, \ldots, x_{n, Y_n}$ be their locations so that
 $$ \xi_n \setminus \xi_{n - 1} = \{x_{n, 1}, x_{n, 2}, \ldots, x_{n, Y_n} \}. $$
 Then, generation~$n + 1$ is defined as follows:
\begin{itemize}
 \item For each vertex~$y$ adjacent to~$x_{n, 1}$, the first particle of generation~$n$ gives birth to a particle sent to~$y$ if and only if~$y$
       is empty and edge~$(x_{n, 1}, y)$ is open. \vspace*{4pt}
 \item For each vertex~$y$ adjacent to~$x_{n, 2}$, the second particle of generation~$n$ gives birth to a particle sent to~$y$ if and only if~$y$
       is empty and edge~$(x_{n, 2}, y)$ is open. \vspace*{4pt}
 \item $\cdots$ \vspace*{4pt}
 \item For each vertex~$y$ adjacent to~$x_{n, Y_n}$, the~$Y_n$th particle of generation~$n$ gives birth to a particle sent to~$y$ if and only if~$y$
       is empty and edge~$(x_{n, Y_n}, y)$ is open.
\end{itemize}
 Note that two particles~$i$ and~$j$ with~$i < j$ might share a common neighbor~$y$ in which case a child of particle~$i$ sent to~$y$ prevents
 particle~$j$ from giving birth onto~$y$.
 For a construction of the birth process from a realization of bond percolation on the dodecahedron, we refer to Figure~\ref{fig:birth}.
 The process is designed so that particles ultimately occupy the open cluster starting at~$x$.
 In particular, the total number of particles equals the cluster size, as proved in the next lemma.
\begin{figure}[t!]
\centering
\scalebox{0.50}{\input{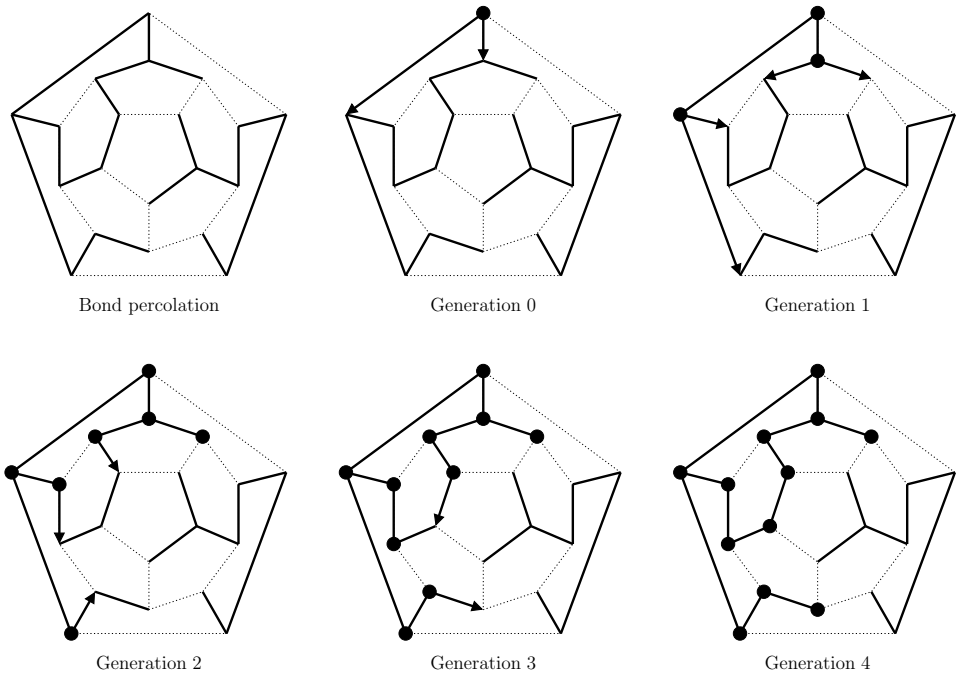}}
\caption{\upshape{Example of a construction of the birth process from a realization of bond percolation (top left picture) on the dodecahedron.
                  The thick lines represent the open edges, the black dots represent the vertices occupied by a particle at each generation, an the
                  arrows represent the birth events, from parent to children.}}
\label{fig:birth}
\end{figure}
\begin{lemma} --
\label{lem:wet-particle}
 The cluster size is given by
 $$ S = \card (\C_x) = \card (\xi_{N - 1}) = Y_0 + Y_1 + \cdots + Y_{N - 1} \quad \hbox{where} \quad N = \card (\V). $$
\end{lemma}
\begin{proof}
 To begin with, we observe that
\begin{itemize}
 \item Because particles can only give birth to another particle sent to an empty vertex, each vertex is ultimately occupied by at most one particle. \vspace*{4pt}
 \item The open cluster containing~$x$ can be written as
       $$ \begin{array}{rcl}
          \C_x = \{y \in \V & \n : \n & \hbox{there is a self-avoiding path of} \\ & \n \n & \hbox{open edges connecting vertex~$x$ and vertex~$y$} \}. \end{array} $$
 \item The set of vertices occupied by a particle of generation~$n$ is
       $$ \begin{array}{rcl}
          \xi_n \setminus \xi_{n - 1} = \{y \in \C_x & \n : \n & \hbox{the shortest self-avoiding path of} \\ & \n \n & \hbox{open edges connecting~$x$ and~$y$ has length~$n$} \}. \end{array} $$
\end{itemize}
 These three properties imply that all the vertices in the open cluster~$\C_x$ are ultimately occupied by exactly one particle whereas the vertices
 outside the cluster remain empty therefore
\begin{equation}
\label{eq:wet-particle-1}
\begin{array}{rcl}
  S = \card (\C_x) & \n = \n & \displaystyle \card (\xi_0) + \card \bigg(\bigcup_{n = 1}^{\infty} \,(\xi_n \setminus \xi_{n - 1}) \bigg) \\
                   & \n = \n & \displaystyle \card (\xi_0) + \sum_{n = 1}^{\infty} \,\card (\xi_n \setminus \xi_{n - 1}) = \sum_{n = 0}^{\infty} \,Y_n. \end{array}
\end{equation}
 In addition, because the graph has~$N$ vertices, the shortest self-avoiding path on this graph must have at most~$N - 1$ edges, from which it follows that
\begin{equation}
\label{eq:wet-particle-2}
  \xi_n = \xi_{n - 1} \quad \hbox{and} \quad Y_n = \card (\xi_n \setminus \xi_{n - 1}) = 0 \quad \hbox{for all} \quad n > N.
\end{equation}
 Combining~\eqref{eq:wet-particle-1} and~\eqref{eq:wet-particle-2} gives the result.
\end{proof} \\


\noindent{\bf Coupling with a branching process.}
 The next step is to compare the number of particles in the birth process with the number of individuals in a branching process~$(X_n)$.
 The process coincides with the birth process when the graph is a tree and is defined by
 $$ X_0 = 1 \quad \hbox{and} \quad X_{n + 1} = X_{n, 1} + X_{n, 2} + \cdots + X_{n, X_n} \quad \hbox{for all} \quad n \geq 0 $$
 where the random variables~$X_{n, i}$ representing the offspring distribution (number of offspring of individual~$i$ at time~$n$) are independent
 and have probability mass function
 $$ X_{0, 1} = \binomial (D, p) \quad \hbox{and} \quad X_{n, i} = \binomial (D - 1, p) \quad \hbox{for all} \quad n, i \geq 1. $$
 This branching process can be visualized as the number of particles in the birth process above modified so that births onto already occupied vertices are allowed.
 In particular, the branching process dominates stochastically the birth process.
\begin{lemma} --
\label{lem:branching-particle}
 For all~$n \geq 0$, we have the stochastic domination~$Y_n \preceq X_n$.
\end{lemma}
\begin{proof}
 As for the branching process, for all~$n \geq 0$ and~$i \leq Y_n$, we let
 $$ Y_{n, i} = \hbox{\# offspring of the~$i$th particle of generation~$n$ in the birth process}. $$
 Because the edges are independently open with the same probability~$p$ and there are exactly~$D$ edges starting from each vertex, the number
 of offspring of the first particle is
\begin{equation}
\label{eq:branching-particle-1}
  Y_1 = Y_{0, 1} = \binomial (D, p).
\end{equation}
 For each subsequent particle, say the particle located at~$z$, we distinguish two types of edges starting from~$z$ just before the particle gives birth.
\begin{itemize}
 \item There are~$m$ edges~$(z, y)$ that are connected to an occupied vertex~$y$.
       Because parent and offspring are located on adjacent vertices, we must have~$m \geq 1$. \vspace*{4pt}
 \item There are~$D - m$ edges~$(z, y)$ that are connected to an empty vertex~$y$.
       These edges have not been used yet in the construction of the birth process, i.e., there has been no previous attempt to give birth through these
       edges, therefore each of these edges is open with probability~$p$ independently of the past of the process.
\end{itemize}
 From the previous two properties, we deduce that, for all~$n > 0$ and~$i \leq Y_n$,
\begin{equation}
\label{eq:branching-particle-2}
\begin{array}{rcl}
  P (Y_{n, i} \geq k) & \n = \n &
  E (P (Y_{n, i} \geq k \,| \,Y_{0, 1}, Y_{1, 1}, \ldots, Y_{n, i - 1})) \vspace*{4pt} \\ & \n \leq \n &
  P (\binomial (D - 1, p) \geq k) = P (X_{n, i} \geq k). \end{array}
\end{equation}
 The stochastic domination follows from~\eqref{eq:branching-particle-1} and~\eqref{eq:branching-particle-2}.
\end{proof} \\


\noindent{\bf Number of individuals.}
 It directly follows from Lemmas~\ref{lem:wet-particle} and~\ref{lem:branching-particle} that
\begin{equation}
\label{eq:wet-branching}
  E (S^k) = E ((Y_0 + Y_1 + \cdots + Y_{N - 1})^k) \leq E ((X_0 + X_1 + \cdots + X_{N - 1})^k)
\end{equation}
 for all~$k > 0$.
 In view of~\eqref{eq:wet-branching}, the last step to complete the proof of Theorem~\ref{th:branching} is to show that the upper bounds in the theorem
 are in fact the first and second moments of the total number of individuals up to generation~$R = N - 1$ in the branching process:
 $$ E (\bar X_R) \quad \hbox{and} \quad E (\bar X_R^2) \quad \hbox{where} \quad \bar X_R = X_0 + X_1 + \cdots + X_R. $$
 The rest of this section is devoted to computing these moments.
\begin{lemma} --
\label{lem:branching-first}
 Let~$\nu = (D - 1) p$. Then,
 $$ E (\bar X_R) = 1 + Dp \bigg(\frac{1 - \nu^R}{1 - \nu} \bigg) \quad \hbox{for all} \quad R > 0. $$
\end{lemma}
\begin{proof}
 For~$i = 1, 2, \ldots, X_1$, let
 $$ \begin{array}{rcl}
    \bar Z_i & \n = \n & \hbox{number of descendants of the~$i$th offspring of the first individual} \vspace*{2pt} \\ &&
                         \hbox{up to generation~$R$, including the offspring}. \end{array} $$
 Then~$\bar X_R = 1 + \bar Z_1 + \cdots + \bar Z_{X_1}$ and the~$\bar Z_i$ are independent of~$X_1$ so
 $$ \begin{array}{rcl}
      E (\bar X_R) & \n = \n & E (E (\bar X_R \,| \,X_1)) = E (E (1 + \bar Z_1 + \cdots + \bar Z_{X_1} \,| \,X_1)) \vspace*{4pt} \\
                   & \n = \n & E (1 + X_1 E (\bar Z_i)) = 1 + E (X_1) E (\bar Z_i) = 1 + Dp E (\bar Z_i). \end{array} $$
 Because~$\bar Z_i$ is the number of individuals up to generation~$R - 1$ in a branching process with offspring
 distribution~$\binomial (D - 1, p)$, we deduce from~\cite[Theorem~2]{jevtic_lanchier_2020} that
 $$ E (\bar X_R) = 1 + Dp \bigg(\frac{1 - (\mu p)^R}{1 - \mu p} \bigg)
                 = 1 + Dp \bigg(\frac{1 - \nu^R}{1 - \nu} \bigg) \quad \hbox{where} \quad \nu = \mu p = (D - 1) p. $$
 This completes the proof.
\end{proof} \\ \\
 Using the same decomposition as in the previous lemma, we now compute the second moment of the number of individuals up to generation~$R = N - 1$.
\begin{lemma} --
\label{lem:branching-second}
 Let~$\nu = (D - 1) p$. Then, for all~$R > 0$,
 $$ E (\bar X_R^2) = \bigg(1 + Dp \bigg(\frac{1 - \nu^R}{1 - \nu} \bigg) \bigg)^2 +
                     \frac{Dp (1 - p)}{(1 - \nu)^2} \bigg(\frac{(1 - \nu^R)(1 + \nu^{R + 1})}{1 - \nu} - 2R \nu^R \bigg). $$
\end{lemma}
\begin{proof}
 Using again~$\bar X_R = 1 + \bar Z_1 + \cdots + \bar Z_{X_1}$ and independence, we get
\begin{equation}
\label{eq:branching-second-1}
  \begin{array}{rcl}
    E (\bar X_R^2) & \n = \n & E (E ((1 + \bar Z_1 + \cdots + \bar Z_{X_1})^2 \,| \,X_1)) \vspace*{4pt} \\
                   & \n = \n & E (E (1 + 2 (\bar Z_1 + \cdots + \bar Z_{X_1}) + (\bar Z_1 + \cdots + \bar Z_{X_1})^2 \,| \,X_1)) \vspace*{4pt} \\
                   & \n = \n & E (1 + 2 X_1 E (\bar Z_i) + X_1 E (\bar Z_i^2) + X_1 (X_1 - 1)(E (Z_i))^2) \vspace*{4pt} \\
                   & \n = \n & 1 + 2 E (X_1) E (\bar Z_i) + E (X_1) E (\bar Z_i^2) + E (X_1 (X_1 - 1))(E (Z_i))^2. \end{array}
\end{equation}
 In addition, using that~$X_1 = \binomial (D, p)$, we get
\begin{equation}
\label{eq:branching-second-2}
  \begin{array}{rcl}
    E (X_1 (X_1 - 1)) & \n = \n & \var (X_1) + (E (X_1))^2 - E (X_1) \vspace*{4pt} \\
                      & \n = \n & Dp (1 - p) + D^2 p^2 - Dp = D (D - 1) p^2. \end{array}
\end{equation}
 Combining~\eqref{eq:branching-second-1} and~\eqref{eq:branching-second-2} gives
 $$ \begin{array}{rcl}
      E (\bar X_R^2) & \n = \n & 1 + 2 Dp E (\bar Z_i) + Dp E (\bar Z_i^2) + D (D - 1) p^2 (E (\bar Z_i))^2 \vspace*{4pt} \\
                     & \n = \n & 1 + 2 Dp E (\bar Z_i) + Dp (\var (\bar Z_i) + (E (\bar Z_i))^2) + D (D - 1) p^2 (E (\bar Z_i))^2 \vspace*{4pt} \\
                     & \n = \n & 1 + 2 Dp E (\bar Z_i) + Dp (Dp + 1 - p)(E (\bar Z_i))^2 + Dp \var (\bar Z_i) \vspace*{4pt} \\
                     & \n = \n & (1 +  Dp E (\bar Z_i))^2 + Dp (1 - p)(E (\bar Z_i))^2 + Dp \var (\bar Z_i). \end{array} $$
 Then, applying~\cite[Theorem~2]{jevtic_lanchier_2020} with~$\mu = D - 1$ and~$\sigma^2 = 0$, we get
 $$ \begin{array}{rcl}
      E (\bar X_R^2) = \bigg(1 & \n + \n & \displaystyle Dp \bigg(\frac{1 - \nu^R}{1 - \nu} \bigg) \bigg)^2 + Dp (1 - p) \bigg(\frac{1 - \nu^R}{1 - \nu} \bigg)^2 \vspace*{8pt} \\
                               & \n + \n & \displaystyle Dp \ \frac{\nu (1 - p)}{(1 - \nu)^2} \bigg(\frac{1 - \nu^{2R - 1}}{1 - \nu} - (2R - 1) \nu^{R - 1} \bigg). \end{array} $$
 Observing also that
 $$ \begin{array}{rcl}
    \displaystyle Dp (1 - p) \bigg(\frac{1 - \nu^R}{1 - \nu} \bigg)^2 & \n + \n &
    \displaystyle Dp \ \frac{\nu (1 - p)}{(1 - \nu)^2} \bigg(\frac{1 - \nu^{2R - 1}}{1 - \nu} - (2R - 1) \nu^{R - 1} \bigg) \vspace*{8pt} \\ & \n = \n &
    \displaystyle \frac{Dp (1 - p)}{(1 - \nu)^2} \bigg(\frac{(1 - \nu)(1 - \nu^R)^2 + \nu (1 - \nu^{2R - 1})}{1 - \nu} - (2R - 1) \nu^R \bigg) \vspace*{8pt} \\ & \n = \n &
    \displaystyle \frac{Dp (1 - p)}{(1 - \nu)^2} \bigg(\frac{1 - 2 \nu^R + 2 \nu^{R + 1} - \nu^{2R + 1}}{1 - \nu} - (2R - 1) \nu^R \bigg) \vspace*{8pt} \\ & \n = \n &
    \displaystyle \frac{Dp (1 - p)}{(1 - \nu)^2} \bigg(\frac{1 - 2 \nu^R + 2 \nu^{R + 1} - \nu^{2R + 1} + (1 - \nu) \nu^R}{1 - \nu} - 2R \nu^R \bigg) \vspace*{8pt} \\ & \n = \n &
    \displaystyle \frac{Dp (1 - p)}{(1 - \nu)^2} \bigg(\frac{(1 - \nu^R)(1 + \nu^{R + 1})}{1 - \nu} - 2R \nu^R \bigg) \end{array} $$
 completes the proof.
\end{proof} \\ \\
 Theorem~\ref{th:branching} directly follows from~\eqref{eq:wet-branching}, and from Lemmas~\ref{lem:branching-first} and~\ref{lem:branching-second}.


\section{Proof of Theorem~\ref{th:plarge}}
\label{sec:plarge}
 Theorem~\ref{th:plarge} relies on the following simple observation:
 vertex~$y \neq x$ cannot be in the percolation cluster starting at~$x$ when all the edges incident to~$y$ are closed.
 In contrast with the comparison with branching processes, this result leads to a good approximation of the moments of the size distribution when
 the probability~$p$ approaches one.
 To prove the theorem, note that
\begin{equation}
\label{eq:plarge-1}
\begin{array}{rcl}
  E (S^k) = \displaystyle E \bigg(\sum_{y \in \V} \,\ind \{y \in \C_x \} \bigg)^k & \n = \n &
            \displaystyle \sum_{y_1, \ldots, y_k \in \V} E (\ind \{y_1 \in \C_x \} \ \cdots \ \ind \{y_k \in \C_x \}) \vspace*{4pt} \\ & \n = \n &
            \displaystyle \sum_{y_1, \ldots, y_k \in \V} P (x \leftrightarrow y_1, \ldots, x \leftrightarrow y_k) \end{array}
\end{equation}
 for all integers~$k$.
 To estimate the last sum, we let~$B_y$ be the event that all the edges incident to~$y$ are closed.
 Using that there are exactly~$D$ edges incident to each vertex, and that there is at most one edge connecting any two different vertices, say~$y \neq z$, we get
\begin{equation}
\label{eq:plarge-2}
\begin{array}{rcl}
           P (B_y) & \n = \n & (1 - p)^D \vspace*{4pt} \\
  P (B_y \cup B_z) & \n = \n & P (B_y) + P (B_z) - P (B_y \cap B_z) \geq 2 (1 - p)^D - (1 - p)^{2D - 1}. \end{array}
\end{equation}
 In addition, we have the inclusion of events
\begin{equation}
\label{eq:plarge-3}
  B_y \subset \{x \not \leftrightarrow y \} \quad \hbox{for all} \quad y \neq x.
\end{equation}
 Combining~\eqref{eq:plarge-2} and~\eqref{eq:plarge-3}, we get
\begin{equation}
\label{eq:plarge-4}
  P (x \not \leftrightarrow y \ \hbox{or} \ x \not \leftrightarrow z) \geq \left\{\begin{array}{lcl}
    (1 - p)^D & \hbox{when} & \card \{x, y, z \} = 2 \vspace*{4pt} \\
  2 (1 - p)^D - (1 - p)^{2D - 1} & \hbox{when} & \card \{x, y, z \} = 3. \end{array} \right.
\end{equation}
 Using~\eqref{eq:plarge-1} with~$k = 1$ and~\eqref{eq:plarge-4}, we deduce that
 $$ \begin{array}{rcl}
      E (S) & \n = \n & \displaystyle 1 + \sum_{y \neq x} \,P (x \leftrightarrow y) = 1 + \sum_{y \neq x} \,(1 - P (x \not \leftrightarrow y)) \vspace*{4pt} \\
            & \n \leq \n & \displaystyle 1 + \sum_{y \neq x} \,(1 - (1 - p)^D) = 1 + (N - 1)(1 - (1 - p)^D) = N - (N - 1)(1 - p)^D. \end{array} $$
 Similarly, applying~\eqref{eq:plarge-1} with~$k = 2$, observing that
 $$ \begin{array}{rcl}
    \card \{(y, z) \in \V^2 : \card \{x, y, z \} = 2 \} & \n = \n & 3 (N - 1) \vspace*{4pt} \\
    \card \{(y, z) \in \V^2 : \card \{x, y, z \} = 3 \} & \n = \n & (N - 1)(N - 2), \end{array} $$
 and using~\eqref{eq:plarge-4}, we deduce that
 $$ \begin{array}{rcl}
      E (S^2) & \n \leq & \n 1 + 3 (N - 1)(1 - (1 - p)^D) + (N - 1)(N - 2)(1 - 2 (1 - p)^D + (1 - p)^{2D - 1}) \vspace*{4pt} \\
              & \n = \n & N^2 - 3 (N - 1)(1 - p)^D - (N - 1)(N - 2)(2 (1 - p)^D - (1 - p)^{2D - 1}) \vspace*{4pt} \\
              & \n = \n & N^2 - (N - 1)(2N - 1)(1 - p)^D + (N - 1)(N - 2)(1 - p)^{2D - 1}. \end{array} $$
 This completes the proof of Theorem~\ref{th:plarge}.


\section{Proof of Theorems~\ref{th:04}--\ref{th:12-20} (inclusion-exclusion identity)}
\label{sec:inclusion-exclusion}
 Theorems~\ref{th:04}--\ref{th:12-20} follow from an application of the inclusion-exclusion identity.
 To begin with, we prove a result (see~\eqref{eq:inclusion-exclusion-4} below) that holds not only for all five Platonic solids but also a larger class of finite
 regular graphs.
 Fix a vertex~$x \in \V$, let~$r$ be the radius of the graph, and define
 $$ \Lambda_s = \{y \in \V : d (x, y) = s \} \quad \hbox{and} \quad N_s = \card (\Lambda_s) \quad \hbox{for} \quad s = 0, 1, \ldots, r. $$
 At least for the Platonic solids, $N_s$ does not depend on the choice of~$x$. Fixing
 $$ y_s \in \Lambda_s \quad \hbox{for all} \quad s = 0, 1, \ldots, r, $$
 and applying~\eqref{eq:plarge-1} with~$k = 1$, we get
\begin{equation}
\label{eq:inclusion-exclusion-1}
  E (S) = \sum_{y \in \V} \,P (x \leftrightarrow y) = \sum_{s = 0}^r \ \sum_{y \in \Lambda_s} P (x \leftrightarrow y) = \sum_{s = 0}^r \,N_s P (x \leftrightarrow y_s).
\end{equation}
 To compute the probabilities~$p_s = P (x \leftrightarrow y_s)$, we label the edges~$0, 1, \ldots, n - 1$, think of each self-avoiding path~$\pi$ as the
 collection of its edges, and let
 $$ \begin{array}{rcl}
    \pi_1 (y_s), \ldots, \pi_{K_s} (y_s) & \n = \n & \hbox{all the self-avoiding paths~$x \to y_s$} \vspace*{4pt} \\
                                     A_i & \n = \n & \hbox{the event that~$\pi_i (y_s)$ is an open path for} \ i = 1, 2, \ldots, K_s. \end{array} $$
 Because the edges are independently open with the same probability~$p$,
 $$ \begin{array}{rcl}
      P (A_{i_1} \cap \cdots \cap A_{i_j}) & \n = \n & P (\pi_{i_1} (y_s), \ldots, \pi_{i_j} (y_s) \ \hbox{are open paths}) \vspace*{4pt} \\
                                           & \n = \n & P (e \ \hbox{is open for all} \ e \in \pi_{i_1} (y_s) \cup \cdots \cup \pi_{i_j} (y_s)) \vspace*{4pt} \\
                                           & \n = \n & p^{\card (\pi_{i_1} (y_s) \,\cup \,\cdots \,\cup \,\pi_{i_j} (y_s))} \end{array} $$
 for all~$0 < i_1 < \cdots < i_j \leq K_s$.
 Here~$\card$ refers to the number of edges in the subgraph that consists of the union of the self-avoiding paths.
 Using that~$x \leftrightarrow y_s$ if and only if at least one of the paths connecting the two vertices is open, and the inclusion-exclusion
 identity, we deduce that
\begin{equation}
\label{eq:inclusion-exclusion-2}
  \begin{array}{rcl}
  \displaystyle P (x \leftrightarrow y_s) & \n = \n &
  \displaystyle P \bigg(\bigcup_{j = 1}^{K_s} \,A_j \bigg) =
  \displaystyle \sum_{j = 1}^{K_s} \ (- 1)^{j + 1} \sum_{0 < i_1 < \cdots < i_j \leq K_s} P (A_{i_1} \cap \cdots \cap A_{i_j}) \vspace*{4pt} \\ & \n = \n &
  \displaystyle \sum_{j = 1}^{K_s} \ (- 1)^{j + 1} \sum_{0 < i_1 < \cdots < i_j \leq K_s} p^{\card (\pi_{i_1} (y_s) \,\cup \,\cdots \,\cup \,\pi_{i_j} (y_s))}. \end{array}
\end{equation}
 Note that, in the previous expression, the index~$j$ corresponds to the number of self-avoiding paths while the second sum is over all possible
 choices of~$j$ paths.
 In particular, the double sum consists in looking at all the possible nonempty sub-collections of the~$K_s$ self-avoiding paths, therefore the right-hand
 side of~\eqref{eq:inclusion-exclusion-2} can be rewritten as
\begin{equation}
\label{eq:inclusion-exclusion-3}
  P (x \leftrightarrow y_s) =
       \sum_{B \subset [K_s] : B \neq \varnothing} \ (- 1)^{\card (B) + 1} \ p^{\card \!\! \big(\bigcup_{i \in B} \pi_i (y_s) \big)}
\end{equation}
 where~$[K_s] = \{1, 2, \ldots, K_s \}$.
 Combining~\eqref{eq:inclusion-exclusion-1} and~\eqref{eq:inclusion-exclusion-3} gives
\begin{equation}
\label{eq:inclusion-exclusion-4}
  E (S) = \sum_{s = 0}^r \ N_s \ \bigg(\sum_{B \subset [K_s] : B \neq \varnothing} \ (- 1)^{\card (B) + 1} \ p^{\card \!\! \big(\bigcup_{i \in B} \pi_i (y_s) \big)} \bigg).
\end{equation}
 The previous equation shows that, at least in theory, computing the mean cluster size reduces to finding the self-avoiding paths that connect any two
 vertices of the graph.
 We now apply~\eqref{eq:inclusion-exclusion-4} to each of the five Platonic solids in order to prove Theorems~\ref{th:04}--\ref{th:12-20}. \\ \\
\begin{proofof}{Theorem~\ref{th:04}}
\begin{figure}[t!]
\centering
\scalebox{0.50}{\input{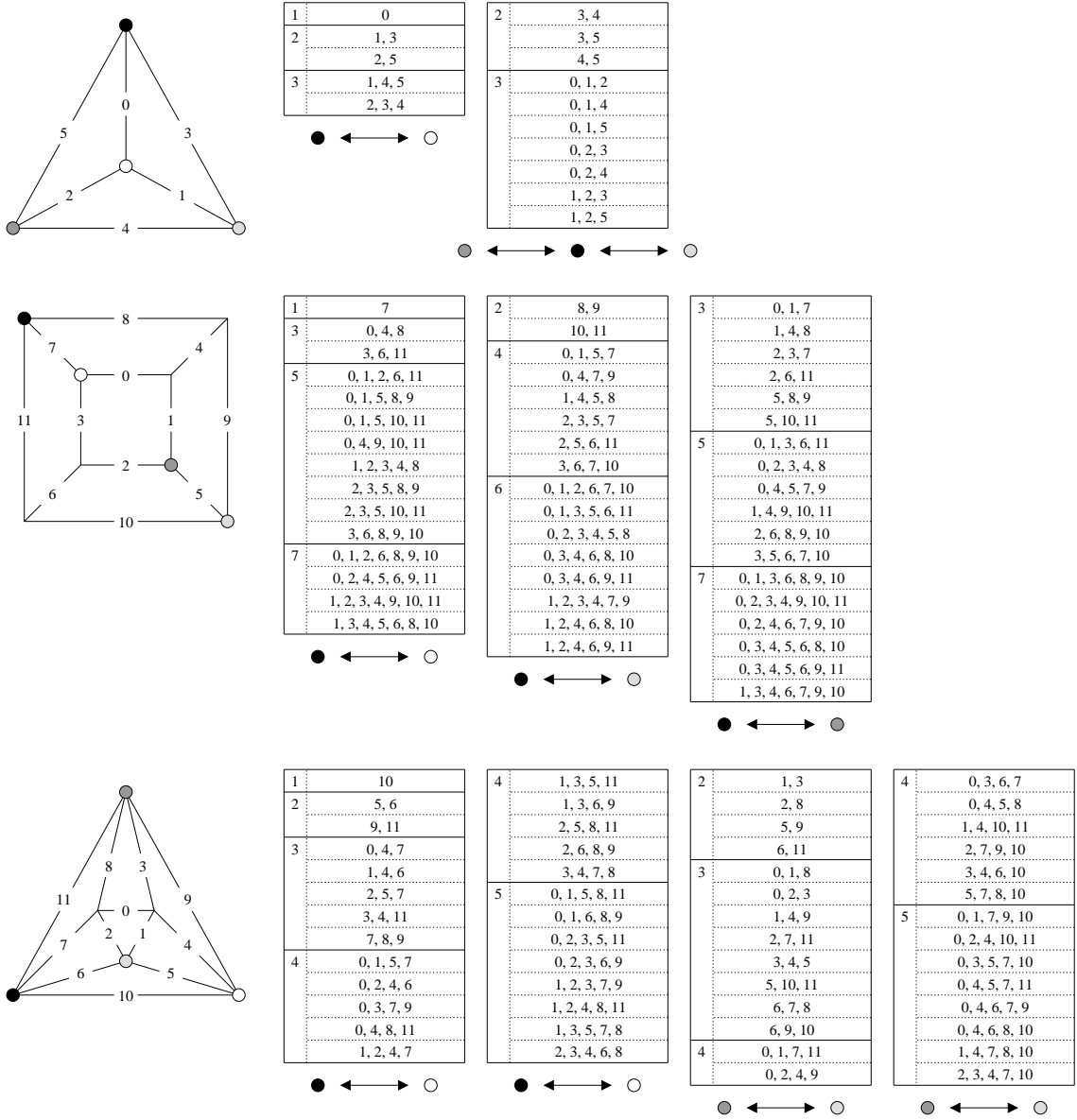}}
\caption{\upshape{The three pictures on the left show planar representations of the tetrahedron, the cube and the octahedron, along with an arbitrary
                  labeling of their edges.
                  The tables on the right give the list of the self-avoiding paths connecting the two vertices (or pairs of self-avoiding paths connecting
                  the three vertices) represented by the black, dark grey, light grey and/or white dots in the pictures.
                  Each path is represented by the collection of its edges using the labels shown in the pictures.
                  The numbers in the first column of each table indicate the length of the paths.}}
\label{fig:path-468}
\end{figure}
 For the tetrahedron, all the vertices are distance one apart and there are exactly five self-avoiding paths connecting any two vertices (see
 first table in Figure~\ref{fig:path-468}).
 Calling these paths~$\pi_1, \ldots, \pi_5$ in the order they are listed in the table, and writing
 $$ \card (\pi_{i_1} \cup \pi_{i_2} \cup \cdots \cup \pi_{i_j}) = |\pi_{i_1, i_2, \ldots, i_j}| $$
 for short, one can easily check that
 $$ \begin{array}{rclrclrclrclrcl}
      |\pi_1| & \n = \n & 1 \quad & |\pi_{1, 2}| & \n = \n & 3 \quad & |\pi_{1, 2, 3}| & \n = \n & 5 \quad & |\pi_{1, 2, 3, 4}| & \n = \n & 6 \quad & |\pi_{1, 2, 3, 4, 5}| & \n = \n & 6 \vspace*{2pt} \\
      |\pi_2| & \n = \n & 2 \quad & |\pi_{1, 3}| & \n = \n & 3 \quad & |\pi_{1, 2, 4}| & \n = \n & 5 \quad & |\pi_{1, 2, 3, 5}| & \n = \n & 6 \vspace*{2pt} \\
      |\pi_3| & \n = \n & 2 \quad & |\pi_{1, 4}| & \n = \n & 4 \quad & |\pi_{1, 2, 5}| & \n = \n & 5 \quad & |\pi_{1, 2, 4, 5}| & \n = \n & 6 \vspace*{2pt} \\
      |\pi_4| & \n = \n & 3 \quad & |\pi_{1, 5}| & \n = \n & 4 \quad & |\pi_{1, 3, 4}| & \n = \n & 5 \quad & |\pi_{1, 3, 4, 5}| & \n = \n & 6 \vspace*{2pt} \\
      |\pi_5| & \n = \n & 3 \quad & |\pi_{2, 3}| & \n = \n & 4 \quad & |\pi_{1, 3, 5}| & \n = \n & 5 \quad & |\pi_{2, 3, 4, 5}| & \n = \n & 5 \vspace*{2pt} \\
              &         &   \quad & |\pi_{2, 4}| & \n = \n & 4 \quad & |\pi_{1, 4, 5}| & \n = \n & 6 \vspace*{2pt} \\
              &         &   \quad & |\pi_{2, 5}| & \n = \n & 4 \quad & |\pi_{2, 3, 4}| & \n = \n & 5 \vspace*{2pt} \\
              &         &   \quad & |\pi_{3, 4}| & \n = \n & 4 \quad & |\pi_{2, 3, 5}| & \n = \n & 5 \vspace*{2pt} \\
              &         &   \quad & |\pi_{3, 5}| & \n = \n & 4 \quad & |\pi_{2, 4, 5}| & \n = \n & 5 \vspace*{2pt} \\
              &         &   \quad & |\pi_{4, 5}| & \n = \n & 5 \quad & |\pi_{3, 4, 5}| & \n = \n & 5 \end{array} $$
 This, together with~\eqref{eq:inclusion-exclusion-3}, implies that, for all~$x \neq y$,
\begin{equation}
\label{eq:tetrahedron-1}
  \begin{array}{rcl}
    P (x \leftrightarrow y) & \n = \n & (p + 2p^2 + 2p^3) - (2p^3 + 7p^4 + p^5) + (9p^5 + p^6) - (p^5 + 4p^6) + p^6 \vspace*{4pt} \\
                            & \n = \n & p + 2p^2 - 7p^4 + 7p^5 - 2p^6 = (0, 1, 2, 0, -7, 7, -2) \cdot P_6. \end{array}
\end{equation}
 Using also~\eqref{eq:inclusion-exclusion-4} and that~$N_1 = 3$ for the tetrahedron, we conclude that
 $$ E (S_4) = 1 + 3 \,(0, 1, 2, 0, -7, 7, -2) \cdot P_6 = (1, 3, 6, 0, -21, 21, -6) \cdot P_6 $$
 which proves the first part of Theorem~\ref{th:04}. \\
\indent To compute the second moment, we observe that any three distinct vertices of the tetrahedron always form a triangle (regardless of the choice
 of the vertices) and, for all~$x \in \V$,
 $$ \begin{array}{rcl}
    \card \{(y, z) \in \V^2 : \card \{x, y, z \} = 2 \} & \n = \n & 3 \times 3 = 9  \vspace*{4pt} \\
    \card \{(y, z) \in \V^2 : \card \{x, y, z \} = 3 \} & \n = \n & 3 \times 2 = 6. \end{array} $$
 Using also~\eqref{eq:plarge-1} with~$k = 2$, we get
\begin{equation}
\label{eq:tetrahedron-2}
  E (S_4^2) = P (x \leftrightarrow x) + 9 P (x \leftrightarrow y) + 6 P (x \leftrightarrow y, x \leftrightarrow z)
\end{equation}
 where vertices~$x, y, z$ are arbitrary but all three distinct.
 In addition, letting~$\gamma_1, \gamma_2, \ldots, \gamma_K$ be the pairs of self-avoiding paths connecting all three vertices, and using the same argument
 as before based on the inclusion-exclusion identity, we get
\begin{equation}
\label{eq:inclusion-exclusion-5}
  P (x \leftrightarrow y, x \leftrightarrow z) =
       \sum_{B \subset [K] : B \neq \varnothing} \ (- 1)^{\card (B) + 1} \ p^{\card \!\! \big(\bigcup_{i \in B} \gamma_i \big)}
\end{equation}
 which can be viewed as the analog of~\eqref{eq:inclusion-exclusion-3}.
 For the tetrahedron, there are~$K = 10$ such paths (see the second table in Figure~\ref{fig:path-468}).
 As previously, computing
 $$ \card \bigg(\bigcup_{i \in B} \gamma_i \bigg) \quad \hbox{for every} \ B \subset [10] = \{1, 2, \ldots, 10 \} $$
 is straightforward in the sense that it does not require any logical thinking.
 However, having ten self-avoiding paths, the sum in~\eqref{eq:inclusion-exclusion-5} is now over
 $$ 2^{10} - 1 = 1,023 \ \ \hbox{terms} $$
 and is therefore unrealistic to compute by hand.
 Also, to compute~\eqref{eq:inclusion-exclusion-5}, we designed a computer program that goes through all the possible subsets~$B \subset [10]$
 and returns six~(= number of edges of the tetrahedron) coefficients~$a_0, a_1, \ldots, a_6$.
 These seven coefficients are initially set to zero and increase or decrease by one according to the following algorithm:
\begin{equation}
\label{eq:algorithm}
\begin{array}{rcl}
  \hbox{replace~$a_j \to a_j + 1$} & \n \hbox{each time} \n & \card \Big(\bigcup_{i \in B} \gamma_i \Big) = j \ \hbox{and} \ \card (B) \ \hbox{is odd}   \vspace*{4pt} \\
  \hbox{replace~$a_j \to a_j - 1$} & \n \hbox{each time} \n & \card \Big(\bigcup_{i \in B} \gamma_i \Big) = j \ \hbox{and} \ \card (B) \ \hbox{is even}. \end{array} 
\end{equation}
 In other words, because the tetrahedron contains six edges, the right-hand side of~\eqref{eq:inclusion-exclusion-5} is a polynomial with degree at most six,
 and the algorithm returns the value of the seven coefficients of this polynomial.
 We point out that the values we obtain are exact because the computer is used to add a large number of integers rather than to simulate the
 percolation process.
 Therefore, the expression of the second moment in the theorem is indeed exact even though we rely on the use of a computer.
 The input of the program is the ten self-avoiding paths represented by the subsets of edges in the second table of Figure~\ref{fig:path-468}, and the output
 of the program is
 $$ a_0 = 0, \quad a_1 = 0, \quad a_2 = 3, \quad a_3 = 5, \quad a_4 = - 18, \quad a_5 = 15, \quad a_6 = - 4. $$
 This, together with~\eqref{eq:tetrahedron-1} and~\eqref{eq:tetrahedron-2}, implies that
 $$ \begin{array}{rcl}
     E (S_4^2) & \n = \n & 1 + 9 \,(0, 1, 2, 0, -7, 7, -2) \cdot P_6 + 6 \,(a_0, a_1, a_2, a_3, a_4, a_5, a_6) \cdot P_6 \vspace*{4pt} \\
               & \n = \n & 1 + 9 \,(0, 1, 2, 0, -7, 7, -2) \cdot P_6 + 6 \,(0, 0, 3, 5, -18, 15, -4) \cdot P_6 \vspace*{4pt} \\
               & \n = \n & (1, 9, 36, 30, -171, 153, -42) \cdot P_6. \end{array} $$
 This completes the proof of Theorem~\ref{th:04}.
\end{proofof} \\ \\
\begin{proofof}{Theorem~\ref{th:06-08}}
\begin{figure}[t!]
\centering
\scalebox{0.60}{\input{poly-06-08.pstex_t}}
\caption{\upshape{Coefficients returned by algorithm~\eqref{eq:algorithm} for the cube (left) and the octahedron (right) using the self-avoiding paths
                  listed in Figure~\ref{fig:path-468}.
                  The last column of each table is equal to the linear combination of the other columns with weight given by the value of the~$N_s$ in
                  the second row, which corresponds to the coefficients of the polynomial in~$p$ equal to the first moment of the cluster size.}}
\label{fig:poly-06-08}
\end{figure}
 The idea is again to compute the sum~\eqref{eq:inclusion-exclusion-4} explicitly by first collecting the self-avoiding paths connecting two vertices and then
 using the computer program mentioned above to obtain the exact value of the coefficients of the polynomial. \vspace*{5pt} \\
\noindent {\bf Cube.}
 For the cube, there are respectively fifteen, sixteen and eighteen self-avoiding paths connecting any two vertices at distance one, two, and three
 from each other, as shown in~Figure~\ref{fig:path-468}.
 Because the cube has twelve edges, the sum consists of a polynomial with degree~12.
 The first four columns in the first table of Figure~\ref{fig:poly-06-08} show the coefficients computed by our program from the list of all the
 self-avoiding paths.
 The first column simply means that, with probability one, a vertex is in the open cluster starting from itself while the second column means that a vertex
 of the cube at distance one of vertex~$x$ is in the open cluster starting at~$x$ with probability
 $$ \begin{array}{l}
      (0, 1, 0, 2, -2, 8, -15, -5, 0, 67, -99, 55, -11) \cdot P_{12} \vspace*{4pt} \\ \hspace*{20pt} = \
       p + 2p^3 -2p^4 + 8p^5 - 15p^6 - 5p^7 + 67p^9 - 99p^{10} + 55p^{11} - 11p^{12}. \end{array} $$
 The second row in the first table of Figure~\ref{fig:poly-06-08} shows the value of~$N_s$ for the cube.
 The last column is simply the linear combination of the first four columns where column~$s$ has weight~$N_s$.
 By~\eqref{eq:inclusion-exclusion-4}, this is the expected value of the cluster size so the proof for the cube is complete. \vspace*{5pt} \\
\noindent {\bf Octahedron.}
 Because the radius of the octahedron is two, two distinct vertices can only be at distance one or two apart.
 There are respectively twenty-six and twenty-eight self-avoiding paths connecting any two vertices at distance one and two from each
 other (see Figure~\ref{fig:path-468}).
 Note that the sum in~\eqref{eq:inclusion-exclusion-3} for two vertices of the octahedron at distance two apart now contains
 $$ 2^{28} - 1 = 268,435,455 \ \ \hbox{terms} $$
 so the use of a computer is absolutely necessary to compute this sum explicitly.
 The sum again consists of a polynomial with degree~12, the common number of edges in the cube and the octahedron,
 and the program gives the coefficients reported in the second table of Figure~\ref{fig:poly-06-08}.
 The rest of the proof is exactly the same as for the cube.
\end{proofof} \\ \\
\begin{proofof}{Theorems~\ref{th:12-20}}
\begin{figure}[t!]
\centering
\scalebox{0.30}{\input{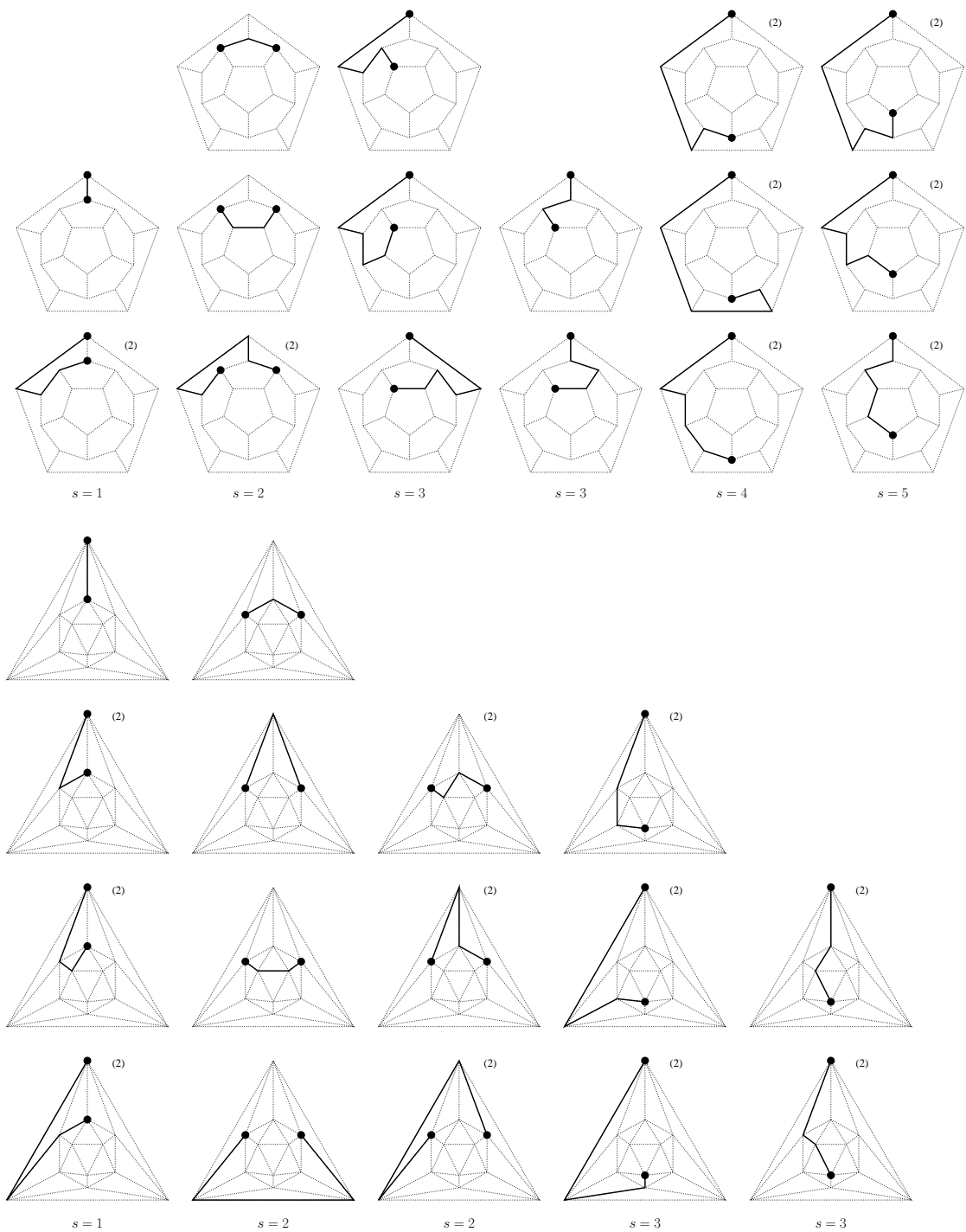}}
\caption{\upshape{Picture of the self-avoiding paths with length at most five connecting two vertices of the dodecahedron at respectively
                  distance~1, 2, 3, 4, and 5, of each other, and picture of the self-avoiding paths with length at most three connecting two
                  vertices of the icosahedron at respectively distance~1, 2, and 3, of each other.
                  The label~(2) next to some pictures means that the mirror image of the path is another path connecting the same two vertices.}}
\label{fig:path-12-20}
\end{figure}
\begin{figure}[t!]
\centering
\scalebox{0.60}{\input{poly-12-20.pstex_t}}
\caption{\upshape{Coefficients returned by algorithm~\eqref{eq:algorithm} for the dodecahedron (left) and the icosahedron (right) using the self-avoiding paths
                  represented in Figure~\ref{fig:path-12-20}.
                  The last column of each table is equal to the linear combination of the other columns with weight given by the value of the~$N_s$ in
                  the second row.
                  Because we only look at a subset of the self-avoiding paths connecting two vertices the last column now gives the coefficients of a
                  polynomial in~$p$ that is smaller than the first moment of the cluster size.}}
\label{fig:poly-12-20}
\end{figure}
 For the dodecahedron and the icosahedron, not only the sum~\eqref{eq:inclusion-exclusion-4} cannot be computed by hand, but also the number
 of self-avoiding paths connecting two vertices is beyond human capability.
 However, we can find lower bounds for the mean cluster size by only taking into account a subset of paths.
 More precisely, given~$x \neq y$, and letting
\begin{itemize}
 \item $\pi_1, \pi_2, \ldots, \pi_J$ be the self-avoiding paths of length~$\leq c$ connecting~$x$ and~$y$, \vspace*{2pt}
 \item $\pi_{J + 1}, \pi_{J + 2}, \ldots, \pi_K$ be the self-avoiding paths of length~$> c$ connecting~$x$ and~$y$,
\end{itemize}
 we deduce from~\eqref{eq:inclusion-exclusion-3} that
\begin{equation}
\label{eq:inclusion-exclusion-6}
  \begin{array}{l}
    P (x \leftrightarrow y) =
         \displaystyle \sum_{B \subset [K] : B \neq \varnothing} \ (- 1)^{\card (B) + 1} \ p^{\card \!\! \big(\bigcup_{i \in B} \pi_i \big)} \vspace*{0pt} \\ \hspace*{100pt} \geq
         \displaystyle \sum_{B \subset [J] : B \neq \varnothing} \ (- 1)^{\card (B) + 1} \ p^{\card \!\! \big(\bigcup_{i \in B} \pi_i \big)}. \end{array}
\end{equation}
 The inequality follows from an inclusion of events:
 if at least one of the first~$J$ paths is open then at least one of the~$K$ paths is open. 
 For both the dodecahedron and the icosahedron, we choose the cutoff~$c$ to be the radius of the graph, meaning that we only consider self-avoiding paths
 with length at most five for the dodecahedron and self-avoiding paths with length at most three for the icosahedron.
 These paths are drawn in Figure~\ref{fig:path-12-20}.
 Because both graphs have thirty edges, the right-hand side of~\eqref{eq:inclusion-exclusion-6} is a polynomial with degree at most~30.
 Fixing a labeling of the edges for both graphs to turn the self-avoiding paths into subsets of~$\{0, 1, \ldots, 29 \}$, and using these
 subsets as inputs, our program returns the values shown in the first table
 of Figure~\ref{fig:poly-12-20} for the dodecahedron and the values shown in the second table for the icosahedron.
 As previously, multiplying each column by the appropriate~$N_s$ listed in the first row of each table gives the coefficients of the polynomial
 on the right-hand side of~\eqref{eq:inclusion-exclusion-6}, which completes the proof of Theorem~\ref{th:12-20}.
 These polynomials have degree less than~30 because we only take into account the shortest self-avoiding paths.
\end{proofof} \\ \\
 In conclusion, using the inclusion-exclusion identity and independence, we proved that computing the expected value of the size of an open cluster
 reduces to finding all the self-avoiding paths connecting two vertices at distance~$1, 2, \ldots, r$ apart.
 Whenever finding all these paths is possible like for the tetrahedron, the cube and the octahedron, our program returns the exact value of the
 coefficients of the polynomial representing the expected value.
 When finding all the paths is not possible like for the dodecahedron and the icosahedron, one can still obtain lower bounds by only looking at
 a subset of self-avoiding paths.


\bibliographystyle{plain}
\bibliography{lanchier_lasalle_2020.bib}

\end{document}